\documentclass[12pt]{amsart}
\usepackage{amssymb,amscd}

\let\frak\mathfrak
\let\Bbb\mathbb

\def\>{\relax\ifmmode\mskip.666667\thinmuskip\relax\else\kern.111111em\fi}
\def\<{\relax\ifmmode\mskip-.333333\thinmuskip\relax\else\kern-.0555556em\fi}
\def\vsk#1>{\vskip#1\baselineskip}
\def\vv#1>{\vadjust{\vsk#1>}\ignorespaces}
\def\vvn#1>{\vadjust{\nobreak\vsk#1>\nobreak}\ignorespaces}

\let\Medskip\medskip
\def\medskip{\par\Medskip}
\let\Bigskip\bigskip
\def\bigskip{\par\Bigskip}

\let\Maketitle\maketitle
\def\maketitle{\Maketitle\thispagestyle{empty}\let\maketitle\empty}

\newtheorem{thm}{Theorem}[section]
\newtheorem{cor}[thm]{Corollary}
\newtheorem{lem}[thm]{Lemma}
\newtheorem{prop}[thm]{Proposition}

\numberwithin{equation}{section}

\theoremstyle{definition}
\newtheorem*{rem}{Remark}
\newtheorem*{example}{Example}

\let\mc\mathcal
\let\nc\newcommand

\nc{\on}{\operatorname}
\nc{\Z}{{\mathbb Z}}
\nc{\C}{{\mathbb C}}
\nc{\N}{{\mathbb N}}
\nc{\pone}{{\mathbb C}{\mathbb P}^1}
\nc{\arr}{\rightarrow}
\nc{\larr}{\longrightarrow}
\nc{\al}{\alpha}
\nc{\W}{{\mc W}}
\nc{\la}{\lambda}
\nc{\su}{\widehat{{\mathfrak sl}}_2}
\nc{\g}{{\mathfrak g}}
\nc{\h}{{\mathfrak h}}
\nc{\m}{{\mathfrak m}}
\nc{\n}{{\mathfrak n}}
\nc{\Gm}{\Gamma}
\nc{\La}{\Lambda}
\nc{\gl}{\widehat{\mathfrak{gl}_2}}
\nc{\bi}{\bibitem}
\nc{\om}{\omega}
\nc{\Res}{\on{Res}}
\nc{\gm}{\gamma}
\nc{\Om}{\Omega}
\nc{\yy}{{\bs y}}
\nc{\kk}{{\bs k}}

\def\z{\mathfrak z}

\def\Res{\on{Res}}

\def\Wr{\on{Wr}}

\def\B{{\mc B}}
\def\D{{\mc D}}

\def\F{{\mc F}}
\def\L{{\mc L}}
\def\M{{\mc M}}

\def\V{{\mc V}}

\let\Dl\Delta

\let\si\sigma
\let\Si\Sigma

\let\der\partial

\let\geq\geqslant

\let\leq\leqslant

\nc{\gln}{\mathfrak{gl}_N}
\nc{\sln}{\mathfrak{sl}_N}

\def\beq{\begin{equation}}
\def\eeq{\end{equation}}
\def\be{\begin{equation*}}
\def\ee{\end{equation*}}

\nc{\bean}{\begin{eqnarray}}
\nc{\eean}{\end{eqnarray}}
\nc{\bea}{\begin{eqnarray*}}
\nc{\eea}{\end{eqnarray*}}
\nc{\bs}{\boldsymbol}
\nc{\Ref}[1]{{\rm(\ref{#1})}}

\nc{\R}{\Bbb R}
\nc{\glN}{\mathfrak{gl}_N}
\nc{\glNt}{\mathfrak{gl}_N[t]}
\nc{\s}{sing}

\nc{\Oml}{{\Om_{\bs\la}}}
\nc{\OmLb}{{\Om_{\bs\La,\bs\la,\bs b}}}
\nc{\Ol}{{\mc O_{\bs\la}}}
\nc{\OLb}{{\mc O_{\bs\La,\bs\la,\bs b}}}
\nc{\Ml}{{\mc M_{\bs\la}}}
\nc{\Mlb}{{\mc M_{\bs\La,\bs\la,\bs b}}}
\nc{\Blb}{{\B_{\bs\La,\bs\la,\bs b}}}
\nc{\Omn}{{\Omega_{\bs n,\bs b,\bs K}}}
\nc{\Omlb}{{\bar\Om_{\bs\la}}}
\nc{\VSl}{{(\V^S)_{\bs\la}}}

\nc{\Dlb}{\Dl_{\bs\La,\bs\la,\bs b,\bs K}}
\nc{\ep}{\epsilon}
\nc{\Vn}{{V^{\otimes n}}}
\nc{\Il}{{\mc I_{\bs\la}}}
\nc{\bla}{{\bs\la}}
\nc{\Fla}{\F_{\bs\la}}
\nc{\GL}{{GL_n(\C)}}
\nc{\ga}{\gamma}
\nc{\Ga}{\Gamma}

\def\Bb{{\mc B}}
\nc{\Nn}{{\mc N}}
\nc\Ll{{\mc L}}
\def\Wr{\on{Wr}}
\nc{\PCN}{{   (\C[x])^2   }}
\nc{\slt}{{\frak{sl}_3}}
\nc\ad{{\on{ad}}}
\nc\gA{{\g(A_2^{(1)})}}
\nc\At{{A_2^{(1)}}}
\nc\Dia{{\on{Diag}}}
\nc\8{\emptyset}
\nc\AT{{A^{(2)}_2}}

\begin{document}

\hrule width0pt
\vsk->

\title[{Critical points  and the mKdV hierarchy  of type $A^{(2)}_2$}]{Critical points
of master functions and the  mKdV hierarchy of type $A^{(2)}_2$}

\author
[A.\,Varchenko,  T.\,Woodruff, D.\,Wright]
{ A.\,Varchenko$\>^{\star}$,  T.\,Woodruff, D.\,Wright}

\maketitle

\begin{center}
{\it Department of Mathematics, University of North Carolina
at Chapel Hill\\ Chapel Hill, NC 27599-3250, USA\/}
\end{center}

{\let\thefootnote\relax
\footnotetext{\vsk-.8>\noindent
$^\star$\,Supported in part by NSF grant DMS--1101508}}

\medskip
\begin{abstract}

We consider the population of critical points
generated from the  critical point of the master function with no variables, which is
associated with the trivial representation of the affine Lie algebra $\AT$.
We describe how the critical points of this population define
rational solutions of the equations of the mKdV hierarchy associated with $\AT$.

\end{abstract}

{\small \tableofcontents

}

\setcounter{footnote}{0}
\renewcommand{\thefootnote}{\arabic{footnote}}

\section{Introduction}

Let $\g$ be a Kac-Moody algebra with invariant scalar product $(\,,\,)$,
 $\h\subset \g$ Cartan subalgebra,
  $\al_0,\dots,\al_N$ simple roots. Let
$\Lambda_1,\dots,\Lambda_n$ be dominant integral weights,
 $k_0,\dots, k_N$
nonnegative integers, $k=k_0+\dots+k_N$.

Consider $\C^n$ with coordinates $z=(z_1,\dots,z_n)$.
Consider $\C^k$ with coordinates $u$ collected into $N+1$ groups, the $j$-th group consisting of $k_j$ variables,
\bea
u=(u^{(0)},\dots,u^{(N)}),
\qquad
u^{(j)} = (u^{(j)}_1,\dots,u^{(j)}_{k_j}).
\eea
The {\it master function} is the multivalued function on $\C^k\times\C^n$ defined by the formula
\bea
&&
\Phi(u,z) = \sum_{a<b} (\La_a,\La_b) \ln (z_a-z_b)
-  \sum_{a,i,j} (\al_j,\La_a)\ln (u^{(j)}_i-z_a) +
\\
&&
+ \sum_{j< j'} \sum_{i,i'} (\al_j,\al_{j'})
\ln (u^{(j)}_i-u^{(j')}_{i'})
+  \sum_{j} \sum_{i<i'} (\al_j,\al_{j})
\ln (u^{(j)}_i-u^{(j)}_{i'}),
\eea
with singularities at the places where the arguments of the logarithms are equal to zero.

Examples of master functions associated with  $\g=\frak{sl}_2$
were considered by Stieltjes and Heine in 19th century, see \cite{Sz}.
Master functions were introduced in \cite{SV}
to construct integral representations for solutions of the KZ equations, see also \cite{V1, V2}.
The critical points of master functions with respect to $u$-variables were used  to
find eigenvectors in the associated Gaudin models by the Bethe ansatz method,
see   \cite{BF, RV, V3}. In important cases the algebra of functions on the critical set of master functions
is closely related to Schubert calculus, see \cite{MTV}.

In \cite{ScV, MV1} it was observed that the critical points of master functions with respect to the $u$-variables
can be deformed and
form families. Having one critical point, one can construct a family of new critical points. The family is
called a population of critical points. A point of the population is a critical point of the same master function
or of another master function associated with the same $\g,\La_1,\dots,\La_n$ but with  different integer parameters
$k_0,\dots,k_N$. The population is a variety isomorphic to the flag variety of the Kac-Moody algebra $\g^t$ Langlands dual
to $\g$, see \cite{MV1, MV2, F}.

In \cite{VW}, it was discovered that the population originated from the critical point of the master function associated
with the affine Lie algebra $\widehat{\frak{sl}}_{N+1}$ and the parameters $n=0, k_0=\dots=k_N=0$ is connected with the mKdV
integrable hierarchy associated with $\widehat{\frak{sl}}_{N+1}$. Namely, that population can be naturally embedded into the
space of $\widehat{\frak{sl}}_{N+1}$ Miura opers so that the image of the embedding is invariant with respect
to all mKdV flows on the space of Miura opers. For $N=1$, that result follows from  the classical paper by M.\,Adler and J.\,Moser \cite{AM}.

In this  paper we prove the analogous statement  for the twisted affine Lie algebra $\AT$.

\medskip
In Sections \ref{sec KM} and \ref{sec MKDV}, we follow the paper \cite{DS} by V.\,Drinfled and V.\,Sokolov  and review
the Lie algebras of types $\AT$, $\At$ and the associated mKdV hierarchies.

In Section \ref{sec gene} formula  \Ref{Master}, we introduce our master functions associated with $\AT$,
\bean
\phantom{aaa}
\Phi(u; k_0,k_1) =
2  \sum_{i<i'}
\ln (u^{(0)}_i-u^{(0)}_{i'}) + 8 \sum_{i<i'}
\ln (u^{(1)}_i-u^{(1)}_{i'})
- 4 \sum_{i,i'}\ln (u^{(0)}_i-u^{(1)}_{i'}) .
\notag
\eean
  Following \cite{MV1, MV2, VW}, we describe
the generation procedure of new critical points starting from a given one. We define the population of critical points
generated from the critical point of the function with no variables, namely, the function corresponding to the parameters $k_0=k_1=0$.
That population is partitioned into complex cells $\C^m$ labeled by finite sequences $J=(j_1,\dots,i_m)$, $m\geq 0$, of the form
$(0,1,0,1,\dots)$ or $(1,0,1,0,\dots)$. Such sequences are called basic.

In Section \ref{sec cr and Miu}, to every basic sequence $J$ we assign a map $\mu^J : \C^m \to \mc M(\AT)$ of that cell to the space
$\mc M(\AT)$ of Miura opers of type $\AT$. We describe some properties of that map.

In Section \ref{sec Vector fields}, we formulate and prove our main result. Theorem \ref{thm main} says that for any basic sequence,
the variety $\mu^J(\C^m)$ is invariant with respect to all mKdV flows on $\mc M(\AT)$ and that variety is point-wise fixed by all
flows $\frac\der{\der t_r}$ with index $r$ greater than $3m+1$.

In the next papers we plan to extend this result to arbitrary affine Lie algebras.

\section{Kac-Moody algebra of types $A_2^{(2)}$ and $\At$}
\label{sec KM}

In this section we follow \cite[Section 5]{DS}.

\subsection{Kac-Moody algebra of type $A_2^{(2)}$}

\subsubsection{Definition}
\label{sec def}

 Consider the Cartan matrix of type $A_2^{(2)}$,
\bea
\AT= \left(
\begin{matrix}
a_{0,0} & a_{0,1} \\
a_{1,0} & a_{1,1}
\end{matrix}
\right)=\left(
\begin{matrix}
2 & -1 \\
-4 & 2
\end{matrix}
\right) .
\eea
The diagonal matrix $D=\on{diag}(d_0,d_2)=\on{diag}(4,1)$ is such that $B=DA=\left(
\begin{matrix}
8 & -4 \\
-4 & 2
\end{matrix}
\right)$ is symmetric.
The Kac-Moody algebra $\g=\g(A^{(2)}_2)$ {\it of type} $A_2^{(2)}$ is the Lie algebra with {\it canonical generators} $e_i,h_i,f_i\in\g, i=0,1$,
subject to the relations
\bea
[e_i,f_j]=\delta_{i,j}h_i,
\qquad
[h_i, e_j] = a_{i,j} e_j , \qquad
[h_i,  f_j] = - a_{i,j} f_j\,,
\\
(\on{ad} e_i)^{1-a_{i,j}} e_j=0,
\qquad
(\on{ad}  f_i)^{1-a_{i,j}} f_j=0,
\qquad
2h_0 + h_1 = 0,
\eea
see this definition in \cite[Section 5]{DS}. More precisely, we have
\bea
[h_0, e_0] =  2e_0,
\qquad
[h_0, e_1] = - e_1,
&\qquad &
[h_1, e_0] = - 4 e_0,
\qquad
[h_1, e_1] =  2e_1,
\\
{}[h_0, f_0] = -2f_0,
\qquad
[h_0, f_1] = f_1,
&\qquad&
[h_1, f_0] = 4 f_0,
\qquad
[h_1, f_1] = - 2f_1,
\\
(\on{ad} e_0)^{2} e_1=0,
\qquad
(\on{ad} e_1)^{5} e_0=0,
&\qquad&
(\on{ad} f_0)^{2} f_1=0,
\qquad
(\on{ad} f_1)^{5} f_0=0 .
\eea

The Lie algebra $\g$ is graded with respect to the {\it standard grading}, $\deg e_i=1, \deg f_i=-1$,\ $ i=0,1$.
Let $\g^j=\{x\in\g\ |\ \deg x=j\}$, then $\g = \oplus_{j\in\Z}\,\g^j$.

Notice that $\g^0$ is the one-dimensional space generated by $h_0,h_1$.
Denote $\h=\g^0$. Introduce elements $\al_0,\al_1$ of the dual space $\h^*$ by the conditions
$\langle \al_j, h_i\rangle =a_{i,j}$ for $i,j=0,1$.
More precisely,
\bea
\langle \al_0, h_0 \rangle =2,
\qquad
 \langle \al_1, h_1 \rangle = 2,
\qquad
\langle \al_0, h_1 \rangle =-4,
\qquad
\langle \al_1, h_0 \rangle =-1
\eea

\subsubsection{Realizations of $\g$}
\label{Realizations of g}

Consider the complex Lie algebra $\frak{gl}_3$ with standard basis $e_{i,j}$,\ $i,j=1,2,3$.

Let $w=e^{\pi i/3}$.
Define the {\it Coxeter automorphism} $C : \frak{sl}_3\to \frak{sl}_3 $  of order 6 by the formula
\bea
e_{1,2} \mapsto w^{-1} e_{2,3}, \qquad e_{2,3} \mapsto w^{-1} e_{1,2},
\qquad
e_{2,1} \mapsto w e_{3,2},
\qquad
e_{3,2} \mapsto w e_{2,1}.
\eea
Denote $(\frak{sl}_3)_j=\{ x \in \frak{sl}_3\ | \ Cx=w^jx\}$.
Then
\bea
&
(\frak{sl}_3)_0 = \langle e_{1,1}-e_{3,3}\rangle,
\quad
(\frak{sl}_3)_1 = \langle e_{2,1} + e_{3,2}, e_{1,3} \rangle, \quad
(\frak{sl}_3)_2 = \langle e_{1,2}-e_{2,3}\rangle,
&
\\
&
(\frak{sl}_3)_3 = \langle e_{1,1}-2e_{2,2}+e_{3,3}\rangle,
\quad
(\frak{sl}_3)_4 = \langle e_{2,1}-e_{3,2}\rangle,
\quad
(\frak{sl}_3)_5 = \langle e_{1,2}+e_{2,3}, e_{3,1} \rangle.
&
\eea
The twisted Lie subalgebra $L(\frak{sl}_3, C)\subset \slt[\xi,\xi^{-1}]$ is the subalgebra
\bea
L(\frak{sl}_3, C) = \oplus_{j \in \mathbb{Z}}\, \xi^{j} \otimes (\frak{sl}_3)_{j\, \on{mod}\, 6} .
\eea
The isomorphism $\tau_C: \g \to L(\frak{sl}_3, C)$ is defined by the formula
\bea
&&
e_0\mapsto \xi\otimes e_{1,3},
\qquad
e_1\mapsto \xi\otimes (e_{2,1}+e_{3,2}),
\\
&&
f_0\mapsto \xi^{-1}\otimes e_{3,1},
\qquad
f_1\mapsto \xi^{-1}\otimes (2e_{1,2}+2e_{2,3}),
\\
&&
h_0\mapsto 1\otimes (e_{1,1}-e_{3,3}),
\qquad
h_1\mapsto 1\otimes (-2e_{1,1}+2e_{3,3}).
\eea
Under this isomorphism we have $\g^j=\xi^j\otimes (\slt)_j$.

Define the {\it standard automorphism} $\sigma_0 :\frak{sl}_3\to \frak{sl}_3 $  of order 2 by the formula
\bea
e_{1,2} \mapsto e_{2,3},
\qquad
e_{2,3} \mapsto e_{1,2},
\qquad
e_{2,1} \mapsto  e_{3,2},
\qquad
e_{3,2} \mapsto  e_{2,1}.
\eea
Let $(\slt)_{0,j}=\{ x\in \slt\ | \ \si_0x=(-1)^jx\}$. Then
\bea
(\slt)_{0,0}=\langle e_{1,1}-e_{3,3}, e_{1,2} + e_{2,3}, e_{2,1}+e_{3,2}\rangle,
\qquad
(\slt)_{0,1}=\langle e_{1,2}- e_{2,3}, e_{2,1} - e_{3,2},  e_{1,3},  e_{3,1}\rangle.
\eea
The twisted Lie subalgebra $L(\frak{sl}_3, \si_0)\subset \slt[\la,\la^{-1}]$ is the subalgebra
\bea
L(\frak{sl}_3, \si_0) = \oplus_{j \in \mathbb{Z}}\, \la^{j} \otimes (\frak{sl}_3)_{0,\,j\, \on{mod}\, 2} .
\eea
The isomorphism $\tau_0: \g \to L(\frak{sl}_3, \si_0)$ is defined by the formula
\bea
&&
e_0\mapsto \la \otimes e_{1,3},
\qquad
e_1\mapsto 1 \otimes (e_{2,1}+e_{3,2}),
\\
&&
f_0\mapsto \la^{-1}\otimes e_{3,1},
\qquad
f_1\mapsto 1 \otimes (2e_{1,2}+2e_{2,3}),
\\
&&
h_0\mapsto 1\otimes (e_{1,1}-e_{3,3}),
\qquad
h_1\mapsto 1\otimes (-2e_{1,1}+2e_{3,3}).
\eea

Define the {\it standard automorphism} $\sigma_1 :\frak{sl}_3\to \frak{sl}_3 $  of order 4,
\bea
e_{1,2} \mapsto i e_{2,3}, \qquad
e_{2,3} \mapsto ie_{1,2},
\qquad
e_{2,1} \mapsto i^{-1} e_{3,2},
\qquad
e_{3,2} \mapsto i^{-1} e_{2,1},
\eea
where $i=\sqrt{-1}$.
Let $(\slt)_{1,j}=\{ x\in \slt\ | \ \si_0x=i^jx\}$. Then
\bea
&&
(\slt)_{1,0}=\langle e_{1,3}, e_{3,1}, e_{1,1}-e_{3,3}\rangle,
\qquad
(\slt)_{1,1}=\langle e_{1,2}- e_{2,3}, e_{2,1} + e_{3,2}\rangle,
\\
&&
(\slt)_{1,2}=\langle e_{1,1}-2e_{2,2}+e_{3,3}\rangle,
\qquad
(\slt)_{1,3}=\langle e_{2,1}- e_{3,2}, e_{1,2} + e_{2,3}\rangle.
\eea
The twisted Lie subalgebra $L(\frak{sl}_3, \si_1)\subset \slt[\mu,\mu^{-1}]$ is the subalgebra
\bea
L(\frak{sl}_3, \si_1) = \oplus_{j \in \mathbb{Z}} \,\mu^{j} \otimes (\frak{sl}_3)_{1,\,j\, \on{mod}\, 4} .
\eea
The isomorphism $\tau_1: \g \to L(\frak{sl}_3, \si_1)$ is defined by the formula
\bea
&&
e_0\mapsto 1 \otimes e_{1,3},
\qquad
e_1\mapsto \mu \otimes (e_{2,1}+e_{3,2}),
\\
&&
f_0\mapsto 1\otimes e_{3,1},
\qquad
f_1\mapsto \mu^{-1} \otimes (2e_{1,2}+2e_{2,3}),
\\
&&
h_0\mapsto 1\otimes (e_{1,1}-e_{3,3}),
\qquad
h_1\mapsto 1\otimes (-2e_{1,1}+2e_{3,3}).
\eea

The composition isomorphism $ L(\frak{sl}_3, \si_0)\to  L(\frak{sl}_3, C)$ is given by the formula
$\la^m\otimes e_{k,l} \mapsto \xi^{3m+k-l}\otimes e_{k,l}$.
The composition isomorphism $ L(\frak{sl}_3, \si_0)\to  L(\frak{sl}_3, \si_1)$ is given by the formula
$\la^m\otimes e_{k,l} \mapsto \mu^{2m+l-k}\otimes e_{k,l}$.

\begin{rem}
The standard automorphisms $\si_0,\si_1$ correspond to the two vertices of the Dynkin diagram of type $A^{(2)}_2$,
see \cite[Section 5]{DS}.
\end{rem}

\subsubsection{ Element $\La$}

Denote by $\La$ the element $e_0+e_1\in \g$. Then $\z=\{x\in\g\ | \ [\La,x]=0\}$ is an Abelian Lie subalgebra of $\g$. Denote $\z^j = \z\cap \g^j$, then
$\z =\oplus_{j\in\Z}\,\z^j$. We have $\dim \z^j=1$ if $j=1,5$ mod $6$ and $\dim \z^j=0$
otherwise.

If $\g$ is realized as $L(\slt,C)$, then the
 element $\xi^{6m+1}\otimes (e_{2,1}+e_{3,2}+ e_{1,3})$
generates $\z^{6m+1}$ and the element
 $\xi^{6m-1}\otimes (e_{1,2}+e_{2,3}+ e_{3,1})$
generates $\z^{6m-1}$.

If $\g$ is realized as $L(\slt,\si_0)$, then the
 element $\la^{2m}\otimes (e_{2,1}+e_{3,2}) + \xi^{2m+1}\otimes e_{1,3}$
generates $\z^{6m+1}$ and the element
 $\la^{2m}\otimes (e_{2,1}+e_{3,2})+ \la^{2m-1}\otimes e_{3,1})$
generates $\z^{6m-1}$.

\begin{lem}
For any $m\in \Z$, the elements
\bea
(\tau_C)^{-1}(\xi^{6m+1}\otimes (e_{2,1}+e_{3,2}+ e_{1,3})),
\qquad
(\tau_0)^{-1}(\la^{2m}\otimes (e_{2,1}+e_{3,2}) + \xi^{2m+1}\otimes e_{1,3})
\eea
 of $\z^{6m+1}$
are equal. Similarly,  the elements
\bea
(\tau_C)^{-1}(\xi^{6m-1}\otimes (e_{1,2}+e_{2,3}+ e_{3,1})),
\qquad
(\tau_0)^{-1}(\la^{2m}\otimes (e_{2,1}+e_{3,2})+ \la^{2m-1}\otimes e_{3,1})
\eea
 of $\z^{6m-1}$ are equal.
\qed
\end{lem}

Denote the elements $(\tau_C)^{-1}(\xi^{6m+1}\otimes (e_{2,1}+e_{3,2}+ e_{1,3}))$,
$(\tau_C)^{-1}(\xi^{6m-1}\otimes (e_{1,2}+e_{2,3}+ e_{3,1}))$ of $\g$ by $\La_{6m+1}$ and
$\La_{6m-1}$, respectively.  Notice that $\La_1=e_0+e_1=\La$.

We set $\La_j=0$ if $j\neq 1,5$ mod $6$.

\begin{lem}
Let us consider the elements $\xi^{6m+1}\otimes (e_{2,1}+e_{3,2}+ e_{1,3})$,
$\xi^{6m-1}\otimes (e_{1,2}+e_{2,3}+ e_{3,1})$  as  $3\times 3$ matrices,
\bea
A_{6m+1}=\left(
\begin{matrix}
0 & 0 & \xi^{6m+1} \\
\xi^{6m+1} & 0 & 0 \\
0 & \xi^{6m+1} & 0
\end{matrix}
\right),
\qquad
A_{6m-1}=\left(
\begin{matrix}
0 & \xi^{6m-1} & 0 \\
0 & 0 & \xi^{6m-1} \\
\xi^{6m-1} & 0 & 0
\end{matrix}
\right),
\eea
respectively. Then  $A_{6m+1} = (A_1)^{6m+1}$ and $A_{6m-1} = (A_1)^{6m-1}$.

Similarly, let us consider the elements $\la^{2m}\otimes (e_{2,1}+e_{3,2}) + \xi^{2m+1}\otimes e_{1,3},
\la^{2m}\otimes (e_{2,1}+e_{3,2})+ \la^{2m-1}\otimes e_{3,1}$   as  $3\times 3$ matrices,
\bea
B_{6m+1}=\left(
\begin{matrix}
0 & 0 & \la^{2m+1} \\
\la^{2m} & 0 & 0 \\
0 & \la^{2m} & 0
\end{matrix}
\right),
\qquad
B_{6m-1}=\left(
\begin{matrix}
0 & \la^{2m} & 0 \\
0 & 0 & \la^{2m} \\
\la^{2m-1} & 0 & 0
\end{matrix}
\right),
\eea
respectively. Then  $B_{6m+1} = (B_1)^{6m+1}$ and $B_{6m-1} = (B_1)^{6m-1}$.
\qed
\end{lem}

\subsection{Kac-Moody algebra of type $A_2^{(1)}$}

\subsubsection{Definition}

 Consider the Cartan matrix of type $A_2^{(1)}$,
\bea
\At= \left(
\begin{matrix}
a_{0,0} & a_{0,1} & a_{0,2}\\
a_{1,0} & a_{1,1} & a_{1,2}\\
a_{2,0} & a_{2,1} & a_{2,2}\\
\end{matrix}
\right)=
\left(
\begin{matrix}
2 & -1 & -1\\
-1 & 2 & -1\\
-1 & -1 & 2\\
\end{matrix}
\right) .
\eea
The Kac-Moody algebra $\gA$ {\it of type} $A_2^{(1)}$ is the Lie algebra with {\it canonical generators} $E_i,H_i,F_i\in\g,$\ $ i=0,1,2$,
subject to the relations
\bea
[E_i,F_j]=\delta_{i,j}H_i,
\eea
\bea
[H_i, E_j] = a_{i,j} E_j , \qquad
[H_i,  F_j] = - a_{i,j} F_j\,,
\eea
\bea
(\on{ad} E_i)^{1-a_{i,j}} E_j=0,
\qquad
(\on{ad}  F_i)^{1-a_{i,j}} F_j=0,
\eea
\bea
H_0 + H_1 + H_2 = 0,
\eea
see this definition in \cite[Section 5]{DS}.
The Lie algebra $\gA$ is graded with respect to the {\it standard grading}, $\deg E_i=1, \deg F_i=-1,$\ $ i=0,1,2$.
Let $\gA^j=\{x\in\gA\ |\ \deg x=j\}$, then $\gA = \oplus_{j\in\Z}\,\gA^j$.

For $j=0,1,2$, we denote by $\n_j^-\subset \gA$ the Lie subalgebra generated by $F_i$, $i\in\{0,1,2,\},\, i\ne j$.
For example, $\n^-_0$ is generated by $F_1,F_2$.

\subsubsection{Realizations of $\gA$}
Consider the Lie algebra $\slt[\la,\la^{-1}]$. The isomorphism $\tau_0^{(0)} :\gA\to \slt[\la,\la^{-1}]$ is defined by the formula
\bea
&&
E_0\mapsto \la\otimes e_{1,3},
\qquad
E_1\mapsto 1\otimes e_{2,1},
\qquad
E_2\mapsto 1\otimes e_{3,2},
\\
&&
F_0\mapsto \la^{-1}\otimes e_{3,1},
\qquad
F_1\mapsto 1\otimes e_{1,2},
\qquad
F_2\mapsto 1\otimes e_{2,3},
\\
&&
H_0\mapsto 1\otimes (e_{1,1}-e_{3,3}),
\qquad
H_1\mapsto 1\otimes (e_{2,2}-e_{1,1}),
\qquad
H_2\mapsto 1\otimes (e_{3,3}-e_{2,2}).
\eea

Let $\ep=e^{2\pi i/3}$.
Define the {\it Coxeter automorphism of type $\At$}, \  $C^{(1)} : \frak{sl}_3\to \frak{sl}_3 $  of order 3 by the formula
\bea
e_{1,2} \mapsto \ep^{-1} e_{1,2}, \qquad e_{2,3} \mapsto \ep^{-1} e_{2,3},
\qquad
e_{2,1} \mapsto \ep e_{2,1},
\qquad
e_{3,2} \mapsto \ep e_{3,2}.
\eea
Denote $(\frak{sl}_3)^{(1)}_j=\{ x \in \frak{sl}_3\ | \ C^{(1)}x=\ep^j x\}$.
Then
\bea
(\frak{sl}_3)^{(1)}_0 = \langle e_{1,1}-e_{2,2}, e_{2,2}-e_{3,3}\rangle,
\quad
(\frak{sl}_3)^{(1)}_1 = \langle e_{2,1}, e_{3,2}, e_{1,3} \rangle,
\quad
(\frak{sl}_3)^{(1)}_2 = \langle e_{1,2}, e_{2,3}, e_{3,1} \rangle.
\eea
The twisted Lie subalgebra $L(\frak{sl}_3, C^{(1)})\subset \slt[\xi,\xi^{-1}]$ is the subalgebra
\bea
L(\frak{sl}_3, C^{(1)}) = \oplus_{j \in \mathbb{Z}}\, \xi^{j} \otimes (\frak{sl}_3)^{(1)}_{j\, \on{mod}\, 3} .
\eea
The isomorphism $\tau_{C^{(1)}}: \gA \to L(\frak{sl}_3, C^{(1)})$ is defined by the formula
\bea
&&
E_0\mapsto \xi\otimes e_{1,3},
\qquad
E_1\mapsto \xi\otimes e_{2,1},
\qquad
E_2\mapsto \xi\otimes e_{3,2},
\\
&&
F_0\mapsto \xi^{-1}\otimes e_{3,1},
\qquad
F_1\mapsto \xi^{-1}\otimes e_{1,2},
\qquad
F_2\mapsto \xi^{-1}\otimes e_{2,3},
\\
&&
H_0\mapsto 1\otimes (e_{1,1}-e_{3,3}),
\qquad
H_1\mapsto 1\otimes (e_{2,2}-e_{1,1}),
\qquad
H_2\mapsto 1\otimes (e_{3,3}-e_{2,2}).
\eea
Under this isomorphism we have $\gA^j=\xi^j\otimes (\slt)^{(1)}_j$.

The composition isomorphism $ \slt[\la,\la^{-1}] \to  L(\frak{sl}_3, C^{(1)})$ is given by the formula
$\la^m\otimes e_{k,l} \mapsto \xi^{3m+k-l}\otimes e_{k,l}$.

\subsubsection{ Element $\La^{(1)}$}

Denote by $\La^{(1)}$ the element $E_0+E_1+E_2\in \gA$. Then $\z(\At)=\{x\in\gA\ | \ [\La^{(1)},x]=0\}$
 is an Abelian Lie subalgebra of $\g(\At)$. Denote $\z(\At)^j = \z(\At)\cap \gA^j$, then
$\z(\At) =\oplus_{j\in\Z}\,\z(\At)^j$. We have $\dim \z(\At)^j=1$ if $j\neq 0$ mod $3$ and $\dim \z(\At)^j=0$
othewise. If $\gA$ is realized as $ \slt[\la,\la^{-1}]$ or $ L(\frak{sl}_3, C^{(1)})$, then a
 basis of $\z(\At)$ is formed by the matrices $(\La^{(1)})^j$ with $j\neq 0$ mod $3$.

Let $\gA$ be realized as $ \slt[\la,\la^{-1}]$. Consider $\La^{(1)}= e_{2,1}+e_{3,2}+\la\otimes e_{1,3}$ as a
$3\times 3$ matrix depending on the parameter $\la$.
Let $T = \sum_{j=-\infty}^n T_j$ be a formal series with $T_j \in  \gA^j$.
Denote $T^+ = \sum_{j=0}^n T_j,$\ {}\  $ T^- = \sum_{j<0} T_j$.

By \cite[Lemma 3.4]{DS}, we may represent $T$ uniquely in the form $T = \sum_{j=-\infty}^m b_j\,(\La^{(1)})^j$,\  $b_j\in \Dia$, where
$\Dia\subset\frak{gl}_3$ is the space of diagonal $3\times 3$ matrices.
Denote $(T)_\La^+ = \sum_{j=0}^n b_j\,(\La^{(1)})^j,$\ {}\  $ (T)_\La^- = \sum_{j<0} b_j\,(\La^{(1)})^j$.

\begin{lem}
We have $(T)_\La^+= T^+$, $(T)_\La^-=T^-$, $b_0=T^0$.

\end{lem}

\begin{proof}
The isomorphism $ \iota : \slt[\la,\la^{-1}] \to L(\frak{sl}_3, C^{(1)})$ is given by the formula
$\la^m\otimes e_{k,l}\mapsto \xi^{3m+l-k}$.
We have $\iota(b_0) = \iota(b_0^1 e_{1,1}+ b_0^2 e_{2,2}+b_0^3 e_{3,3})=1\otimes (b_0^1 e_{1,1}+ b_0^2 e_{2,2}+b_0^3 e_{3,3})\in \gA^0$,
$\iota(b_1\La^{(1)}) = \iota((b_1^1 e_{1,1}+ b_1^2 e_{2,2}+b_1^3 e_{3,3})(e_{2,1}+e_{3,2}+\la e_{1,3}))
=\iota(b_1^1 \la e_{1,3}+ b_1^2 e_{2,1}+b_1^3 e_{3,2})=\xi\otimes (b_1^1 e_{1,3}+ b_1^2 e_{2,1}+b_1^3 e_{3,2})\in\gA^1$,
$\iota(b_{-1}(\La^{(1)})^{-1}) = \iota((b_{-1}^1 e_{1,1}+ b_{-1}^2 e_{2,2}+b_{-1}^3 e_{3,3})(e_{1,2}+e_{2,3}+\la^{-1} e_{3,1}))
=\iota(b_{-1}^1 e_{1,2}+ b_{-1}^2 e_{2,3}+b_{-1}^3 \la^{-1}e_{3,1})=\xi\otimes (b_{-1}^1 e_{1,2}+ b_{-1}^2 e_{2,3}+b_{-1}^3 e_{3,1})\in\gA^{-1}$.
Similarly one checks that $\iota(b_j\,(\La^{(1)})^j)\in\gA^j$ for any $j$.
\end{proof}

Let us consider the elements $F_0, 2F_1+2F_2$ as the  $3\times 3$ matrices
 $\la^{-1}e_{3,1}, 2e_{1,2}+2e_{2,3}$, respectively.

\begin{lem}
\label{lem exp}
 Let $g \in \C$. Then
 \bean
 \label{formula exp}
 &&
 \\
 \notag
 &&
 e^{gF_0} = 1 + ge_{3,3}(\La^{(1)})^{-1},
 \qquad
 e^{g(2F_1+2F_2)} = 1 + 2g(e_{1,1}+e_{2,2})(\La^{(1)})^{-1} + 2g^2 e_{1,1} (\La^{(1)})^{-2}.
 \eean
 \qed
\end{lem}

\begin{lem}
\label{lem lambda} We have
$(\La^{(1)})^{-1} = e_{1,2}+e_{2,3} + \la^{-1}e_{3,1}$, and
\bean
\label{formula La}
e_{i+1,i+1}\La^{(1)} = \La^{(1)} e_{i,i}, \qquad
e_{i,i}(\La^{(1)})^{-1} = (\La^{(1)})^{-1}e_{i+1,i+1}
\eean
for all $i$, where we set $e_{4,4}=e_{1,1}$.
 \qed
\end{lem}

\subsubsection{ Lie algebra $\g(A^{(2)}_2)$ as a subalgebra of $\gA$}
\label{as a sub}
The map $\varrho: \g(A^{(2)}_2) \to \gA$,
\bea
e_0 \mapsto E_0,
\qquad
e_1 \mapsto E_1+E_2,
\qquad
f_0 \mapsto F_0,
\qquad
f_1 \mapsto 2F_1+2F_2,
\eea
realizes the Lie algebra $\g(A^{(2)}_2)$ as a subalgebra of $\gA$. This embedding preserves the standard grading
and $\varrho(\La)=\La^{(1)}$.  We have $\varrho(\z(A^{(2)}_2))\subset\z(A^{(1)}_2)$.

If there is no confusion, we will consider $\g(\AT)$ as a subalgebra of $\g(\At)$.

\section{mKdV equations}
\label{sec MKDV}

In this section we follow \cite{DS}.

\subsection{mKdV equations of type $A^{(2)}_2$}

Denote by $\Bb$ the space of complex-valued functions of one variable $x$.  Given a finite dimensional vector space $W$, denote by $\Bb(W)$ the space of
$W$-valued functions of $x$. Denote by $\der$ the differential operator $\frac d{dx}$.

Consider the Lie algebra $\tilde \g$ of the formal differential operators of the form $c\der + \sum_{i=-\infty}^n p_i$, $c\in \C, p_i\in \B(\g^i)$.
Let $U=\sum_{i<0} U_i$, $U_i \in \B(\g^{i})$. If $\L\in\tilde\g$, define
\bea
e^{\ad U}(\L) = \L + [U,\L] + \frac 1{2!}[U,[U,\L]]+\dots \ .
\eea
The operator $e^{\ad U}(\L)$ belongs to $\tilde \g$. The map $e^{\ad U}$ is an automorphism of the Lie algebra $\tilde \g$.
The automorphisms of this type form a group.

If elements of $\g$ are realized as matrices depending on a parameter
as in Section \ref{Realizations of g}, then
$e^{\ad U}(\L) = e^U\L e^{-U}$.

A {\it Miura oper} of type $A^{(2)}_2$  is a differential operator of the form
\bean
\label{Miura la}
\L = \der + \Lambda + V
\eean
where $\Lambda = e_0+e_1\in\g$ and $V\in \B(\g^0)$.
Any Miura oper of type $A^{(2)}_2$ is an element of $\tilde \g$.
Denote by $\mc{M}(A^{(2)}_2)$ the space of all Miura opers of type $A^{(2)}_2$.

\begin{prop}
[{\cite[Proposition 6.2]{DS}}]
\label{Prop U}

For any Miura oper $\L$ of type $A^{(2)}_2$ there exists an element
 $U=\sum_{i<0} U_i$, $U_i \in \B(\g^{i})$, such that the operator $\L_0 = e^{\ad U}(\L)$ has the form
\bea
\L_0 = \der + \La + H,
\eea
where $H=\sum_{j<0} H_j, H_j\in \B(\z^{j})$. If $U,\tilde U$ are two such elements, then
$e^{\ad U}e^{-\ad \tilde U} = e^{\ad T}$, where $T = \sum_{j<0} T_j$, $T_j\in \z^j$.

\end{prop}

Let $\L, U$ be as in Proposition \ref{Prop U}. Let  $r=1,5$ mod $6$.
The element $\phi(\La_r)=e^{-\ad U}(\La_r)$ does not depend on the choice of $U$ in Proposition \ref{Prop U}.

The element $\phi(\La_r)$
is of the form $\sum^n_{i=-\infty} \phi(\La_r)^i$, $\phi(\La_r)^i\in B(\g^i)$.
We set $\phi(\La_r)^+ = \sum^n_{i=0} \phi(\La_r)^i$,
$\phi(\La_r)^- = \sum_{i<0} \phi(\La_r)^i$.

\medskip

Let $r\in\Z_{>0}$ and  $r=1,5$ mod $6$.
The differential equation
\bean
\label{mKdVr}
\frac{\partial \L}{\partial t_r}  = [\phi(\La_r)^+, \L]
\eean
 is called the {\it  $r$-th  mKdV equation} of type $A^{(2)}_2$.

Equation \Ref{mKdVr} defines a vector field $\frac\der{\der t_r}$ on the space $\M(A^{(2)}_2)$ of Miura opers. For
all $r,s$, the vector fields $\frac\der{\der t_r}$, $\frac\der{\der t_s}$ commute, see \cite[Section 6]{DS}.

\begin{lem}[{\cite{DS}}]
\label{lem der}

We have
\bean
\label{mKdVr 0}
\frac{\partial\L}{\partial t_r}  = - \frac d{dx}\phi(\La_r)^0.
\eean
\end{lem}

See the proof of Proposition 6.8 in \cite{DS}.

\subsection{mKdV equations of type $A^{(1)}_2$}

A {\it Miura oper} of type $A^{(1)}_2$  is a differential operator of the form
\bean
\label{Miura la}
\L = \der + \Lambda + V
\eean
where $\Lambda = E_0+E_1+E_2\in\gA$ and $V\in \B(\gA^0)$.
Denote by $\mc{M}(A^{(1)}_2)$ the space of all Miura opers of type $A^{(1)}_2$.

\begin{prop}
[{\cite[Proposition 6.2]{DS}}]
\label{Prop U3}

For any Miura oper $\L$ of type $A^{(1)}_2$ there exists an element
$U=\sum_{i<0} U_i$, $U_i \in \B(\gA^{i})$, such that the operator $\L_0 = e^{\ad U}(\L)$ has the form
\bea
\L_0 = \der + \La + H,
\eea
where $H=\sum_{j<0} H_j, H_j\in \B(\z(A^{(1)}_2)^{j})$. If $U,\tilde U$ are two such elements, then
$e^{\ad U}e^{-\ad \tilde U} = e^{\ad T}$, where $T = \sum_{j<0} T_j$, $T_j\in \z(A^{(1)}_2)^j$.

\end{prop}

Let $\L, U$ be as in Proposition \ref{Prop U3}. Let  $r\ne 0$ mod $3$.
The element $\phi((\La^{(1)})^r)=e^{-\ad U}( (\La^{(1)})^r)$ does not depend on the choice of $U$ in Proposition \ref{Prop U3}.

The element $\phi((\La^{(1)})^r)$
is of the form $\sum^n_{i=-\infty} \phi((\La^{(1)})^r)^i$, $\phi((\La^{(1)})^r)^i\in\B(\gA^i)$.
We set $\phi((\La^{(1)})^r)^+ = \sum^n_{i=0} \phi((\La^{(1)})^r)^i$,
$\phi((\La^{(1)})^r)^- = \sum_{i<0} \phi((\La^{(1)})^r)^i$.

\medskip

Let $r\in\Z_{>0}$ and  $r\ne 0$ mod $3$.
The differential equation
\bean
\label{mKdVr 3}
\frac{\partial \L}{\partial t_r}  = [\phi((\La^{(1)})^r)^+, \L]
\eean
 is called the {\it  $r$-th  mKdV equation} of type $A^{(1)}_2$.

Equation \Ref{mKdVr 3} defines a vector field $\frac\der{\der t_r}$ on the space $\M(A^{(1)}_2)$ of Miura opers of type $A^{(1)}_2$. For
all $r,s$, the vector fields $\frac\der{\der t_r}$, $\frac\der{\der t_s}$ commute, see \cite[Section 6]{DS}.

\subsection{Comparison of mKdV equations of types $A^{(2)}_2$ and $A^{(1)}_2$}
\label{sec comp}

Consider $\g(A^{(2)}_2)$ as a Lie subalgebra of $\gA$, see Section \ref{as a sub}. Let $\L$ be a Miura oper
of type $A^{(2)}_2$. Then $\L$ is a Miura oper of type $A^{(1)}_2$.

\begin{lem}
Let $r=1,5$ mod $6$, $r>0$. Let $\L^{A^{(2)}_2}(t_r)$ be the solution of the $r$-th mKdV equation of type $A^{(2)}_2$
with initial condition $\L^{A^{(2)}_2}(0)=\L$. Let $\L^{A^{(1)}_2}(t_r)$ be the solution of the $r$-th mKdV equation of type $A^{(1)}_2$
with initial condition $\L^{A^{(1)}_2}(0)=\L$. Then $\L^{A^{(2)}_2}(t_r)=\L^{A^{(1)}_2}(t_r)$ for all values of $t_r$.
\qed
\end{lem}

\begin{proof}
The element $U$ in Proposition \ref{Prop U} which is used to construct the mKdV equation of type $A^{(2)}_2$ can be used also
to construct the mKdV equation of type $A^{(1)}_2$.
\end{proof}

\subsection{KdV equations of type $A^{(1)}_2$}
\label{sec KdV}

Let $\Bb((\der^{-1}))$ be the algebra of formal pseudodifferential operators of the form
 $a=\sum_{i \in \Z} a_i \der^i$, with $a_i \in \Bb$ and finitely many terms with $i > 0$.
 The relations in this algebra are
\bea
\partial^k u - u \partial^k = \sum_{i = 1}^\infty k(k-1)\dots(k-i+1)\frac{d^i u}{dx^i}\partial^{k-i}
\eea
for any $k\in\Z$ and $u\in\Bb$.
For $a = \sum_{i \in \Z} a_i \der^i \in \Bb((\der^{-1}))$, define $a^+ = \sum_{i \geq 0} a_i \der^i$.

Denote $\Bb[\der] \subset \Bb((\der^{-1}))$ the subalgebra of differential operators $a = \sum_{i=0}^m a_i \der^i$
with $m\in\Z_{\geq 0}$. Denote $\D \subset \Bb[\der]$ the affine subspace of the differential operators of the
form $L=\der^3 + u_1\der + u_0$.

For $L \in \D$, there exists a unique $L^{\frac{1}{3}} =\der + \sum_{i\leq 0}a_i\der^i \in \Bb((\der^{-1}))$
 such that  $(L^{\frac{1}{3}})^3 = L$. For $r\in\N$, we have $L^{\frac{r}{3}} = \der^r + \sum^{r-1}_{i=-\infty}
 b_i\der^i$, $b_i\in \B$.
 We set $(L^{\frac{r}{3}})^+ = \der^r + \sum^{r-1}_{i=0} b_i\der^i$.

For $r\in \Z_{>0}$, the differential equation
\beq
\label{KdVr}
\frac{\partial L}{\partial t_r} = [L,(L^{\frac{r}{3}})^+]
\eeq
is called the  {\it  $r$-th  KdV equation} of type $A^{(1)}_2$.

 Equation \Ref{KdVr} defines flows $\frac{\partial}{\partial t_r}$ on the space $\D$.
  For all $r,s $ the flows $\frac{\partial}{\partial t_r}$ and $\frac{\partial}{\partial t_s}$ commute, see \cite{DS}.

\subsection{Miura maps}
\label{sec Miura maps}

Let $\Ll = \der + \Lambda + V$ be a Miura oper of type $A^{(1)}_2$ with $V = \sum_{k = 1}^3 v_k e_{k,k}$, $\sum_{k=1}^3v_k=0$.  For $i=0,1,2$, define
the scalar differential operator $L_i=\der^3+u_{1,i}\der + u_{0,i}\in\D$ by the formula
\bean
\label{miuramap}
&&
L_0 = (\der - v_3)(\der - v_{2})(\der - v_1),
\quad
L_1 = (\der - v_1)(\der - v_3)(\der - v_2),
\\
&&
\phantom{aaaaaaaaaaaa}
L_2 = (\der - v_2)(\der - v_1)(\der - v_3).
\notag
\eean

\begin{thm}[{\cite[Proposition 3.18]{DS}}]
\label{thm mkdvtokdv}
Let a Miura oper $\Ll$ satisfy the mKdV equation \Ref{mKdVr 3} for some $j$.  Then for every $i=0,1,2,$ the differential operator
$L_i$ satisfies the KdV equation \Ref{KdVr}.
\end{thm}

We define the {\it  $i$-th Miura map} by the formula
\bea
\frak{ m}_i \ : \ \mc M(A^{(1)}_2)\  \to \ \D,
\quad
 \Ll \ \mapsto \ L_i,
 \eea
  see \Ref{miuramap}.

For $i=0,1,2$, an {\it $i$-oper } is a differential operator of the form
\bea
\Ll = \der + \Lambda + V + W
\eea
with $V \in \Bb(\gA^0)$ and $W \in \Bb(\n^-_i)$.
For $w \in \Bb(\n^-_i)$ and an $i$-oper $\Ll$, the differential operator
$e^{\text{ad} \, w}(\L)$
is an $i$-oper.  The $i$-opers $\Ll$ and $e^{\text{ad} \, w} (\Ll)$ are called {\it $i$-gauge equivalent.}
  A Miura oper is an $i$-oper for any $i$.

\begin{thm}
[{\cite[Proposition 3.10]{DS}}]
\label{thm gaugemiura}
If Miura opers $\Ll$ and $\tilde \Ll$ are $i$-gauge equivalent, then
$\frak m_i(\Ll) = \frak m_i(\tilde\Ll)$.
\end{thm}

\section{Critical points of master functions and generation of pairs of polynomials}
\label{sec gene}
In this section we follow \cite{MV1, MV2, VW}.
For functions $f(x),g(x)$,  we denote
\bea
\Wr(f,g) = f(x)g'(x)-f'(x)g(x)
\eea
the Wronskian determinant.

\subsection{Master function}
Choose a pair of nonnegative integers $\bs{k} = (k_0,k_1).$
Consider variables $ u = (u_i^{(j)})$, where $j = 0,1,$ and $i = 1,\dots,k_j$.
 The {\it master function} $\Phi(u; \bs k)$
is defined by the formula
\bean
\label{Master}
\phantom{aaa}
\Phi(u,\bs k) =
2  \sum_{i<i'}
\ln (u^{(0)}_i-u^{(0)}_{i'}) + 8 \sum_{i<i'}
\ln (u^{(1)}_i-u^{(1)}_{i'})
- 4 \sum_{i,i'}\ln (u^{(0)}_i-u^{(1)}_{i'}) .
\eean
The product of symmetric groups
$\Sigma_{\bs k}=\Si_{k_0}\times \Si_{k_1}$ acts on the set of variables
by permuting the coordinates with the same upper index.
The  function $\Phi$ is symmetric with
respect to the $\Si_{\bs k}$-action.

A point $u$ is a {\it critical point} if $d\Phi=0$ at $u$. In other words, $u$ is a critical point if
\bean
\label{Bethe eqn 1}
&&
\sum_{i' \neq i}\frac {2}{ u_i^{(0)} - u_{i'}^{(0)}}
-\sum_{i'=1}^{k_{1}} \frac{4}{ u_i^{(0)} -u_{i'}^{(1)}} = 0, \qquad i=1,\dots,k_0,
\\
\notag
&&
\sum_{i' \neq i}\frac {8}{ u_i^{(1)} - u_{i'}^{(1)}}
-\sum_{i'=1}^{k_{0}} \frac{4}{ u_i^{(1)} -u_{i'}^{(0)}} = 0, \qquad i=1,\dots,k_1.
\eean
The critical set is $\Si_{\bs k}$-invariant.
All orbits have the same cardinality $k_0! k_1!$\ .
We do not make distinction between critical points in the same orbit.

\begin{rem}
The master functions $\Phi(u,\bs k)$ for all vectors $\bs k$ are associated with the Kac-Moody algebra with  Cartan matrix
$\left(
\begin{matrix}
2 & -4 \\
-1 & 2
\end{matrix}
\right)$, which is dual to the Cartan matrix $\AT$, see \cite{SV, MV1, MV2}.
\end{rem}

\subsection{Polynomials representing critical points}
\label{PRCP}

Let $u = (u_i^{(j)})$ be a critical point of the master function
$\Phi$.
Introduce the pair of polynomials $ y=( y_0(x),  y_1(x))$,
\bean\label{y}
y_j(x)\ =\ \prod_{i=1}^{k_j}(x-u_i^{(j)}).
\eean
  Each polynomial is considered up to multiplication
by a nonzero number.
The pair defines a point in the direct product
$\PCN$.
We say that the pair $ y=( y_0(x),  y_1(x))$ {\it represents the
critical point } $u = (u_i^{(j)})$.

It is convenient to think that the pair $\bs{y}^\emptyset = (1, 1)$ of constant polynomials
represents  in $\PCN$ the critical point of the master function with no variables.
This corresponds to the case  $\bs k = (0, 0)$.

We say that a given pair $ y\in \PCN$ is {\it generic} if each polynomial $y_j(x)$ of the pair
has no multiple roots and
the polynomials  $y_0(x)$ and
 $y_{1}(x)$ have no common roots. If a pair represents a critical point, then it is generic,
 see \Ref{Bethe eqn 1}.  For example, the pair  ${y}^\emptyset $ is generic.

\subsection{Elementary generation}
\label{Elementary generation}

The pair is called {\it fertile} if  there exist polynomials $\tilde y_0, \tilde y_1\in \C[x]$ such that
the following two equations are  satisfied,
\bean
\label{wronskian-critical eqn 2}
\Wr(y_0, \tilde y_0) = y_1^4,
\qquad
\Wr(y_1,\tilde y_1) = y_0 .
\eean
These equations  can be written as
\bean
\label{Wr-Cr}
\Wr(y_j, \tilde{y_j}) = \prod_{i\ne j}y_i^{-a_{i,j}} , \qquad j=0,1.
\eean

For  example, the pair $y^\emptyset$ is fertile and $\tilde y_0=x+c_1, \tilde y_1=x+c_2$, where $c_1,c_2$ are arbitrary numbers.

Assume that a pair of polynomials $y=(y_0,y_1)$ is fertile.
Equation $\Wr(y_0,\tilde y_0) = y_1^4$ is a first order inhomogeneous differential equation with respect to $\tilde y_0$.
Its solutions are
\bean
\label{deG 0}
\tilde y_0 = y_0\int \frac{y_1^4}{y_0^2} dx + c y_0,
\eean
where $c$ is any number.
The pairs
\bean
\label{simple 0}
y^{(0)}(x,c) = (\tilde y_0(x,c), y_1(x))
 \in   \PCN \
\eean
 form a one-parameter family.  This family  is called
 the {\it generation  of pairs  from $ y$ in the $0$-th direction}.
 A pair of this family is called an {\it immediate descendant} of $y$ in the $0$-th direction.

Similarly, equation $\Wr(y_1, \tilde y_1) = y_0$ is a first order inhomogeneous differential equation with respect to $\tilde y_1$.
Its solutions are
\bean
\label{deG 1}
\tilde y_1 = y_1\int \frac{y_0}{y_1^2} dx + c y_1,
\eean
where $c$ is any number.
The pairs
\bean
\label{simple 1}
y^{(1)}(x,c) = (y_0(x), \tilde y_1(x,c))
\quad \in  \quad \PCN \
\eean
 form a one-parameter family.  This family  is called
 the {\it generation  of pairs  from $ y$ in the $1$-st direction}.
A pair of this family is called an {\it immediate descendant} of $y$ in the $1$-st direction.

\begin{thm}
[\cite{MV1}]
\label{fertile cor}
${}$

\begin{enumerate}
\item[(i)]
A generic pair $y = (y_0, y_1)$
represents a critical point of a master function
if and only if $y$ is fertile.

\item[(ii)] If $ y$ represents a critical point,
then for any $c\in\C$ the pairs $y^{(0)}(x,c)$ and $y^{(1)}(x,c)$
are fertile.

\item[(iii)]
If $y$ is generic and fertile, then for almost all values of the parameter
 $c\in \C$ both  pairs $y^{(0)}(x,c)$ and $y^{(1)}(x,c)$ are generic.
The exceptions form a finite set in $\C$.

\item[(iv)]

Assume that a sequence ${y}_i, i = 1,2,\dots$ of fertile pairs
has a limit ${y}_\infty$ in $\PCN$ as $i$ tends to infinity.
\begin{enumerate}
\item[(a)]
Then the limiting pair ${y}_\infty$ is fertile.
\item[(b)]
For $j = 0,1$, let ${y}_\infty^{(j)}$ be an immediate
descendant of ${y}_\infty$.
Then for $j=0,1$, there exist immediate descendants
${y}_i^{(j)}$ of ${y}_i$ such that ${y}_\infty^{(j)}$
is the limit of ${y}_i^{(j)}$ as $i$ tends to infinity.

\end{enumerate}

\end{enumerate}
\end{thm}

\subsection{Degree increasing generation}
\label{Degree increasing generation}

Let $y=(y_0,y_1)$ be a generic fertile pair of polynomials.
For $j=0,1$, define $k_j=\deg y_j$.

The polynomial $\tilde y_0$ in \Ref{wronskian-critical eqn 2}
is of degree $k_0$ or
$\tilde k_0=4k_1+ 1 - k_0$. We say that the  generation $(y_0,y_1) \to (\tilde y_0,y_1)$ is
{\it  degree increasing } in the $0$-th direction  if $\tilde k_0 > k_0$. In that
case $\deg \tilde y_0=\tilde k_0$ for all $c$.

If the generation is degree increasing in the $0$-th direction we normalize family
 \Ref{simple 0} and construct a map
$ Y_{y,0} : \C \to (\C[x])^2$ as follows. First we multiply the polynomials $y_0,y_1$ by numbers to make them monic.
 Then we choose the monic polynomial $ y_{0,0}$ satisfying the equation $\Wr(y_0, y_{0,0})$
 $=\, a\,y_1^4$, for some $a\in\C^\times$,
 and such that the coefficient of $x^{k_0}$ in $y_{0,0}$ equals zero.
Such $y_{0,0}$ exists and is unique.
 Set
 \bean
 \label{tilde y0}
 \tilde y_0(x,c)=y_{0,0}(x) + cy_0(x)
 \eean
 and define
\bean
\label{normalized 0}
 Y_{y,0} \ :\ \C\ \to\ (\C[x])^2, \qquad c \mapsto\ y^{(0)}(x,c) = (\tilde y_0(x,c),y_1(x)).
 \eean
Both polynomials of the pair $y^{(0)}(x,c)$ are monic.

The polynomial $\tilde y_1$ in \Ref{wronskian-critical eqn 2}
is of degree $k_1$ or
$\tilde k_1= k_0 + 1 - k_1$. We say that the  generation $(y_0,y_1) \to (y_0,\tilde y_1)$ is
{\it  degree increasing } in the $1$-st direction  if $\tilde k_1 > k_1$. In that
case $\deg \tilde y_1=\tilde k_1$ for all $c$.

If the generation is degree increasing in the $1$-st direction we normalize family
 \Ref{simple 1} and construct a map
$ Y_{y,1} : \C \to (\C[x])^2$ as follows. First we multiply the polynomials $y_0,y_1$ by numbers to make them monic.
 Then we choose the monic polynomial $ y_{1,0}$ satisfying the equation $\Wr(y_1, y_{1,0})$
 $=\,a\,y_0$, for some $a\in\C^\times$,
 and such that the coefficient of $x^{k_1}$ in $y_{1,0}$ equals zero. Such $y_{1,0}$ exists and is unique.
 Set
 \bean
 \label{tilde y1}
 \tilde y_1(x,c)=y_{1,0}(x) + cy_1(x)
 \eean
 and define
\bean
\label{normalized 1}
 Y_{y,1} \ :\ \C\ \to\ (\C[x])^2, \qquad c \mapsto\ y^{(1)}(x,c) = (y_0(x),\tilde y_1(x,c)).
 \eean
Both polynomials of the pair $y^{(1)}(x,c)$ are monic.

\subsection{Degree-transformations and generation of vectors of integers}
\label{sec degree transf}

The degree-transformations
\bean
\label{l-transformation}
&&
\bs k=(k_0,k_1)\ \mapsto \
\bs k^{(0)}=(4k_1+1-k_0,k_1),
\\
\notag
&&
\bs k=(k_0,k_1)\ \mapsto \
\bs k^{(1)}=(k_0,k_0+1-k_1)
\eean
correspond to the shifted action of reflections $ w_0,w_1\in W_{\! A_{2}^{(2)}}$,
where $W_{A_{2}^{(2)}}$ is the Weyl group of type ${A_{2}^{(2)}}$ and $w_0,w_1$
are its standard generators, see Lemma 3.11 in \cite{MV1} for more detail.

We take formula \Ref{l-transformation} as the definition of {\it degree-transformations}:
\bean
\label{s_i l-transf}
&&
w_0\ :\ \bs k=(k_0,k_1)\ \mapsto \
\bs k^{(0)}=(4k_1+1-k_0,k_1),
\\
\notag
&&
w_1\ :\ \bs k=(k_0,k_1)\ \mapsto \
\bs k^{(1)}=(k_0,k_0+1-k_1)
\eean
acting on arbitrary vectors $\bs k=(k_0,k_1)$.

\medskip

We start with the vector $\bs k^\emptyset=(0,0)$ and a sequence $J=(j_1,j_2,\dots, j_m)$ of
integers, where $J= (0,1,0,1,0,1,\dots)$ or  $J=(1,0,1,0,1,0,\dots)$.
We apply the corresponding degree transformations to  $\bs k^\emptyset$ and obtain
the sequence of vectors $\bs k^\emptyset,$\ $  \bs k^{(j_1)} =w_{j_1}\bs k^\emptyset, $\ $
  \bs k^{(j_1,j_2)} = w_{j_2} w_{j_1}\bs k^\emptyset$,\dots,
\bean
\label{gen vector}
\bs k^J  = w_{j_m}\dots w_{j_2} w_{j_1}\bs k^\emptyset .
\eean
We say that the {\it vector $\bs k^J $ is generated from $(0,\dots,0)$ in the direction of $J$}.

\medskip
For example, for $J=(0,1,0,1,0,1)$ we get the sequence
$(0,0),$ $ (1,0),$ $ (1,2),$ $(8,2), (8,7), $ $(21,7), (21,15)$. If
$J=(1,0,1,0,1,0)$, then the sequence  is
$(0,0),$ $ (0,1), (5,1), (5,5),$ $ (16,5), (16,12), (33,12)$.

\medskip
We call a sequence $J$ {\it degree increasing} if for every $i$ the transformation
$w_{j_i}$ applied to  $  w_{j_{i-1}}\dots w_{j_1}\bs k^\emptyset$
increases the $j_i$-th coordinate.

\begin{lem}
\label{lem cyclic increas}
If $J=(0,1,0,1,0,1,\dots)$, then after $2n$ steps of this procedure the degree vector is
$(3n^2-2n,(3n^2+n)/2)$. If $J=(1,0,1,0,1,0,\dots)$, then after $2n+1$ steps of this procedure the degree vector is
$(3n^2+2n,(3n^2+5n+2)/2)$.
\qed
\end{lem}

\begin{cor}
\label{cor deg incr}
Both sequences $(0,1,0,1,0,1,\dots)$ and $(1,0,1,0,1,0,\dots)$ are degree increasing.
\qed
\end{cor}

The sequences $(0,1,0,1,0,1,\dots)$ and $(1,0,1,0,1,0,\dots)$ will be called the {\it basic sequences}.

\subsection{Multistep generation}
\label{sec generation procedure}

Let $J = (j_1,\dots,j_m)$ be a basic sequence.
Starting from $y^\emptyset=(1,1)$ and $J$, we construct by induction on $m$
a map
\bea
Y^J : \C^m \to (\C[x])^2.
\eea
If $J=\emptyset$, the map $Y^\emptyset$ is the map $\C^0=(pt)\ \mapsto y^\emptyset$.
If $m=1$ and $J=(j_1)$,  the map
$Y^{(j_1)} :  \C \to (\C[x])^N$ is given by one of the formulas \Ref{normalized 0} or \Ref{normalized 1}
for $y=y^\emptyset$ and $j=j_1$. More precisely, equation $\Wr(y_0,\tilde y_{0}) = y_1^4$
takes the form $\Wr(1,\tilde y_{0}) =1$. Then $\tilde y_{0,0}= x$ and
\bea
Y^{(0)}\ :\ \C \mapsto (\C[x])^2, \qquad
c \mapsto (x+c, 1).
\eea
 By Theorem \ref{fertile cor}
all pairs in the image are fertile and almost all pairs are generic
(in this example all pairs are generic).  Similarly, equation $\Wr(y_1,\tilde y_{1}) = y_0$
takes the form $\Wr(1,\tilde y_{1}) =1$. Then $\tilde y_{1,0}= x$ and
\bea
Y^{(1)}\ :\ \C \mapsto (\C[x])^2, \qquad
c \mapsto (1,x+c).
\eea

Assume that for ${\tilde J} = (j_1,\dots,j_{m-1})$,  the map
$Y^{{\tilde J}}$ is constructed. To obtain  $Y^J$ we apply the
generation procedure in the $j_m$-th
direction to every pair of the image of $Y^{{\tilde J}}$. More precisely, if
\bean
\label{J'}
Y^{{\tilde J}}\ : \
{\tilde c}=(c_1,\dots,c_{m-1}) \ \mapsto \ (y_0(x,{\tilde c}), y_1(x,{\tilde c})),
\eean
then
\bean
\label{Ja}
&&
Y^{J} : \C^m \mapsto (\C[x])^2, \quad
({\tilde c},c_m) \mapsto
(y_{0,0}(x,{\tilde c}) + c_m y_{0}(x,{\tilde c}), y_1(x,{\tilde c})),\qquad \on{if}\, j_m=0,
\\
\notag
&&
Y^{J} : \C^m \mapsto (\C[x])^2, \quad
({\tilde c},c_m) \mapsto
(y_0(x,{\tilde c}), y_{1,0}(x,{\tilde c}) + c_m y_{1}(x,{\tilde c})),\qquad \on{if}\, j_m=1,
\eean
see formulas \Ref{tilde y0}, \Ref{tilde y1}.
The map  $Y^J$ is called  the {\it generation  of pairs   from $ y^\emptyset$ in the $J$-th direction}.

\begin{lem}
\label{lem gen procedure}
All pairs in the image of $Y^J$ are fertile and almost all pairs are generic. For any $c\in\C^m$
the pair $Y^J(c)$ consists of monic polynomials. The degree vector of this pair
equals $\bs k^J$, see \Ref{gen vector}.
\qed
\end{lem}

\begin{lem}
\label{lem uniqeness}
The map  $Y^J$ sends distinct points of $\C^m$ to distinct points of  $(\C[x])^2$.
\end{lem}

\begin{proof}
The lemma is easily proved by induction on $m$.
\end{proof}

\begin{example}
\label{example AM}
If $J=(0,1)$, then
\bea
Y^{(0)}(c_1)= (x+c_1,1), \qquad
Y^{(0,1)}(c_1,c_2)=
(x+c_1, (x+c_1)^2+ c_2-c_1^2).
\eea
If $J=(1,0)$, then
\bea
Y^{(1)}(c_1)= (1,x+c_1), \qquad
Y^{(1,0)}(c_1,c_2)=
((x+c_1)^4+ c_2-c_1^4, x+c_1).
\eea
\end{example}

The set of all pairs $(y_0,y_1)\in (\C[x])^2$ obtain from $y^\emptyset=(1,1)$
by generations in all degree increasing directions is called the {\it population of pairs}
generated from $y^\emptyset$, c.f. \cite{MV1}.

\section{Critical points of master functions and Miura opers}
\label{sec cr and Miu}

\subsection{Miura oper associated with a pair of polynomials, \cite{MV2}}
Define a map
\bea
\mu : (\C[x])^2 \to \mc M(A^{(2)}_2),
\eea
which sends a pair $y=(y_0,y_1)$ to the Miura oper $\L=\der + \La + V$ with
\bea
V=\ln'\big( \frac{y_1^2}{y_0}\big)\,h_0\,,
\eea
where for a function $f(x)$ we denote $\ln'(f(x))=f'(x)/f(x)$.
We say that the Miura oper $\mu(y)$ is {\it associated to the pair of polynomials} $y$.
For example,
\bea
\L^\emptyset = \der +\La
\eea
is associated to the pair $y^\emptyset=(1,1)$.

We have
\bean
\label{def eq}
\langle \alpha_0, V\rangle = \ln' \big( \frac{y_1^4}{y_0^2}\big) ,
\qquad
\langle \alpha_1, V\rangle = \ln' \big( \frac{y_0}{y_1^2} \big).
\eean
 Equations \Ref{def eq} can be written as
\bean
\label{Def eq}
\langle \alpha_j, V\rangle = \ln' \big( \prod_{i=0}^1 y_i^{-a_{i,j}} \big) ,
\eean
see \cite{MV2}.

\subsection{Deformations of Miura opers of type $\AT$, \cite{MV2}}

\begin{lem}[\cite{MV2}]

Let $\L$ $= \der + \Lambda + V$ be a Miura oper of type $\AT$.  Let $g \in \Bb$. Let $f_j, j \in \{0,1\}$, be one of canonical
generators of $\g(\AT)$, see Section \ref{sec def}.  Then
\bean \label{adeq}
e^{\on{ad} \,gf_j}(\L) = \der + \Lambda + V - g h_j - (g^\prime - \langle \alpha_j , V \rangle g + g^2)f_j.
\eean
\end{lem}

\begin{cor}[\cite{MV2}]
Let $\L = \der + \Lambda + V$ be a Miura oper. Then $e^{\on{ad}\, g f_j } (\L)$
is a Miura oper if and only if the scalar function
$g$ satisfies the Ricatti equation
\bean
\label{Ric}
g' - \langle \alpha_j , V \rangle g +  g^2 = 0 \ .
\eean
\end{cor}

Let $\L = \der + \Lambda + V$ be a Miura oper with $V=vh_0$. Assume that the functions $v$ is a rational functions of $x$.
For $j\in\{0,1\}$, we say that $\L$ is  {\it deformable in the $j$-th direction}
if equation \Ref{Ric} has a nonzero solution $g$, which is a rational function.

\begin{thm} [\cite{MV2}]
\label{ricc thm}
Let the  Miura oper
$\L = \partial  +  \La  +  V$ be associated with a pair of polynomials $y=(y_0, y_1)$.
Let $j\in \{0,1\}$. Then $\L$ is deformable in the $j$-th direction
if and only if there exists a polynomial $\tilde y_j$ satisfying equation
\Ref{Wr-Cr}. Moreover, in that case any nonzero rational
solution $g$ of the Ricatti equation \Ref{Ric} has the form
$g = \mathrm{ln}' (\tilde y_j/ y_j)$ where $\tilde y_j$ is a solution of equation
\Ref{Wr-Cr}. If $g = \mathrm{ln}' (\tilde y_i/ y_i)$, then
the Miura oper
\bean
\label{transformation}
e^{\on{ad}\, g f_i}(\L) = \partial  + \La + V - g  h_j
\eean
is associated  the pair $y^{(j)}$, which is obtained from the pair $y$ by replacing $y_j$ with $\tilde y_j$.

\end{thm}

\begin{proof} Write \Ref{Ric} as
\bean\label{Ric1}
g'/g  +  g  =  \on{ln}'\big( \prod_{j=0}^1 y_i^{- a_{i,j}}
\big) \ .
\eean
If $g$ is a rational function, then $g \to 0$ as $x \to \infty$ and all poles
of $g$ are simple. Moreover, the residue of $g$ at any point is an integer. Hence
$g = c'/c$ for a suitable rational function $c(x)$. Then
\bean\label{Ric2}
c\ = \ \int  \prod_{j=0}^1 y_j(x)^{-a_{i,j}} dx \
\eean
and equation \Ref{Wr-Cr} has a polynomial solution
$\tilde y_j = - c y_j$. Conversely if  equation
\Ref{Wr-Cr} has a polynomial solution $\tilde y_i$, then
the function $c$ in \Ref{Ric2} is rational. Then $g = c'/c$ is a rational solution of
equation \Ref{Ric}.

Let $g = \on{ln}' c =\on{ln}' (\tilde y_i/ y_i)$,
where $\tilde y_i$ is a solution of \Ref{Wr-Cr}. Then
\bea
e^{\on{ad}\, g f_j }(\L) = \partial  + \La  +
V -\on{ln}'(\tilde y_j / y_j) h_j
\eea
and
\bea
\langle \al_k, V \rangle  - \langle   \al_k, h_j \rangle \on{ln}'(\tilde y_j / y_j)\,
 =
 \on{ln}'\Big( \prod_{i=0}^1 y_i^{-a_{i,k}}
\Big)  -  a_{j,k} \on{ln}'(\tilde y_j / y_j)\,
 =    \on{ln}'\Big(\tilde y_j^{-a_{j,k}}
\prod_{i=0, \, i\neq j }^1 y_i^{-a_{i,k}}\Big) .
\eea
\end{proof}

Note that if equation \Ref{Ric} has one nonzero rational solution
$g = c'/c$ with rational $c(x)$, then other nonzero (rational) solutions have the form
$g = c'/(c + \on{const})$.

\subsection{Miura opers associated with the generation procedure}

\label{Miura opers with cr points}

Let $J = (j_1,\dots,j_m)$ be a basic sequence, see Section \ref{sec degree transf}.
Let $Y^J : \C^m \to \PCN $
be the generation of pairs from $y^\emptyset$ in the $J$-th direction.
We define the associated family of Miura opers by the formula:
\bea
\mu^J \  :\ \C^m\ \to\ \mc M(\AT), \qquad c \ \mapsto \ \mu(Y^J(c)) .
\eea
The map $\mu^J$ is called the {\it generation of Miura opers from $\Ll^\emptyset$
in the $J$-th direction.}

For $\ell=1,\dots,m$, denote $J_\ell=(j_1,...,j_\ell)$ the beginning $\ell$-interval of
the sequence $J$. Consider the associated map $Y^{J_\ell} : \C^\ell\to\PCN$. Denote
\bea
Y^{J_\ell}(c_1,\dots,c_\ell) = (y_0(x,c_1,\dots,c_\ell, \ell), y_1(x,c_1,\dots,c_\ell, \ell)).
\eea
Introduce
\bean
\label{g's}
g_1(x,c_1,\dots,c_m) &=&
\ln'(y_{j_1}(x,c_1,1)) ,
\\
\notag
g_\ell(x,c_1,\dots,c_m) &=&
\ln'(y_{j_\ell}(x,c_1.\dots,c_\ell,\ell)) - \ln'(y_{j_\ell}(x,c_1,\dots,c_{\ell-1},\ell-1)),
\eean
for $\ell=2,\dots,m$.
For $c\in\C^m$, define  $U^J(c)=\sum_{i<0} (U^J(c))_i$, $ (U^J(c))_i\in\B(\g^i)$, depending on $c\in\C^m$,
by the formula
\bean
\label{T}
e^{-\,\ad U^J(c)} = e^{\ad g_m(x,c) f_{j_m}} \cdots e^{\ad g_1(x,c) f_{j_1}}.
\eean

\begin{lem}
\label{lem formula} For $c\in\C^m$, we have
\bean
\label{operformula}
\mu^J(c)\ = \ e^{-\,\ad U^J(c)}(\Ll^\8)
\eean
 and
\bean
\label{oper2}
\mu^J(c)\ =\
  \der + \Lambda -\sum_{\ell=1}^m g_\ell(x,c) h_{j_\ell}.
  \eean
\end{lem}

\begin{proof}
The lemma follows from Theorem \ref{ricc thm}.
\end{proof}

\begin{cor}
\label{cor der}
Let $r>0$ and $r=1,5$ mod $6$. Let $c\in\C^m$. Let $\frac {\der\phantom{a}}{\der t_r}\big|_{\mu^J(c)}$ be the value at $\mu^J(c)$
of the vector field of the
$r$-th mKdV flow on the space $\mc M(A^{(2)}_2)$, see \Ref{mKdVr}. Then
\bean
\label{T mkdv}
\frac {\der}{\der t_r}\big|_{\mu^J(c)} = - \frac {\der}{\der x}(e^{-\,\ad U^J(c)}(\La_r))^0.
\eean

\end{cor}

\begin{proof}
The corollary follows from \Ref{mKdVr} and \Ref{operformula}.
\end{proof}

We have a natural  embedding $\mc M(\AT) \hookrightarrow \mc M(\At)$. Let $\frak m_i : \mc M(\At)\to \D,\ \Ll \mapsto L_i,$ be the Miura maps
defined in Section \ref{sec Miura maps} for  $i=0,1$. Below we consider the composition of the embedding and a Miura map.

Denote ${\tilde J}=(j_1,\dots,j_{m-1})$. Consider the associated family
$\mu^{{\tilde J}} : \C^{m-1}\to\mc M(\AT)$. Denote ${\tilde c}=(c_1,\dots,c_{m-1})$.

\begin{lem}
\label{lem j j'}
For all $({\tilde c},c_m)\in\C^m$, we have  $\frak m_1\circ \mu^J({\tilde c},c_m) = \frak m_1\circ \mu^{{\tilde J}}({\tilde c})$ if $j_m=0$ and we have
 $\frak m_0\circ \mu^J({\tilde c},c_m) = \frak m_0\circ \mu^{{\tilde J}}({\tilde c})$ if $j_m=1$.

\end{lem}

\begin{proof}
The lemma follows from formula \Ref{operformula} and Theorem \ref{thm gaugemiura}.
\end{proof}

\begin{lem}
\label{lem der c-m}
If $j_m=0$, then
\bean
\label{form der c-m 0}
\frac{\der \mu^J}{\der c_m}({\tilde c},c_m) \ =\ -a\
\frac{y_{1}(x,{\tilde c},m-1)^4}
{y_{0}(x,{\tilde c},c_m,m)^2} \,h_{0}
\eean
for some  $a\in\C^\times$. If $j_m=1$, then
\bean
\label{form der c-m 1}
\frac{\der \mu^J}{\der c_m}({\tilde c},c_m) \ =\ -a\
\frac{y_{0}(x,{\tilde c},m-1)}
{y_{1}(x,{\tilde c},c_m,m)^2} \,h_{1}
\eean
for some $a\in\C^\times$.
\end{lem}

\begin{proof}
Let $j_m=0$. Then $ y_{0}(x,{\tilde c},c_m,m)=y_{0,0}(x,{\tilde c}) + c_m y_{0}(x,{\tilde c},m-1),$
where $y_{0,0}(x,{\tilde c})$ is such that
 \bea
 \Wr (y_{0}(x,{\tilde c},m-1), y_{0,0}(x,{\tilde c})) \ =\ a\ y_{1}(x,{\tilde c},m-1)^4,
 \eea
 for some $a\in\C^\times$, see \Ref{tilde y0}.
We have $ g_m = \ln'(y_{0}(x,{\tilde c},c_m,m))- \ln' (y_{0}(x,{\tilde c},m-1))$.

By formula \Ref{oper2}, we have
\bea
\frac{\der \mu^J}{\der c_m}({\tilde c},c_m) = -\frac{\der g_m}{\der c_m}({\tilde c},c_m) \,h_0
\eea
and
\bea
&&
\frac{\der g_m}{\der c_m}({\tilde c},c_m) = \frac \der{\der c_m}\left(
\frac{y_{0,0}'(x,{\tilde c}) + c_m y_{0}'(x,{\tilde c},m-1)}{y_{0,0}(x,{\tilde c}) + c_m y_{0}(x,{\tilde c},m-1)}
\right)=
\\
&&
\phantom{aaaa}
=
\frac{\Wr (y_{0,0}(x,{\tilde c}), y_{0}(x,{\tilde c},m-1))}{(y_{0,0}(x,{\tilde c}) + c_m y_{0}(x,{\tilde c},m-1))^2}
=a\
\frac{y_{1}(x,{\tilde c},m-1)^4}
{y_{0}(x,{\tilde c},c_m,m)^2} \,.
\eea
This proves formula \Ref{form der c-m 0}. Formula \Ref{form der c-m 1} is proved similarly.
\end{proof}

Let us describe the kernels of the differentials of the  Miura maps $\frak{m}_i, i=0,1$, restricted to Miura opers of type $\AT$.
A Miura oper $\L=\der + \La + vh_0$ of type $\AT$ is mapped to the differential operator
$(\der + v)\der(\der-v) = \der^3-(2v'+v^2)\der -(v''+vv')$ by the Miura map $\frak{m}_0$ and to  the differential operator
$(\der - v)(\der+v)\der = \der^3+(v'-v^2)\der$ by the Miura map $\frak{m}_1$. The derivative maps are
\bea
d\frak{m}_0 &:& Xh_0 \mapsto -(2X'+ 2vX)\der - (X'' + vX'+v'X),
\\
d\frak{m}_1 &:& Xh_0 \mapsto  (X' - 2vX)\der,
\eea
where $x\in\B$.

\begin{lem}
\label{lem kernel}
Assume that $\L$ is associated to a pair  $(y_0,y_1)$, that is,   $v=\ln'\big(\frac{y_1^2}{y_0}\big)$.
Then the kernel of $d\frak{m}_0$ at $\L$ is one-dimensional and is generated by the function $\frac{y_0}{y_1^2}h_0$.
Also the kernel of $d\frak{m}_1$ at $\L$ is one-dimensional and is generated by the function $\frac{y_1^4}{y_0^2}h_0$.

\end{lem}

\begin{proof}

We have $X\in \on{ker}\,d\frak{m}_0$ if and only if $X'+ vX=0$. This implies the first statement.
Similarly,  $X\in \on{ker}\,d\frak{m}_1$ if and only if $X'-2vX=0$.
This implies the second statement.
\end{proof}

\section{Vector fields}
\label{sec Vector fields}

\subsection{Statement}
\label{sec statement}

Let $r>0$ and $r=1,5$ mod $6$. Recall that  we denote by $ \frac \der{\der t_r}$ the $r$-th mKdV vector field
on the space  $\mc M(\AT)$  of Miura opers of type $\AT$. We also denote by $ \frac \der{\der t_r}$ the $r$-th mKdV vector field of type $\At$
on the space  $\mc M(\At)$  of Miura opers of type $\At$.
We have a natural embedding $\mc M(\AT) \hookrightarrow \mc M(\At)$.
Under this embedding the vector $ \frac \der{\der t_r}$  on $\mc M(\AT)$ equals the vector field $ \frac \der{\der t_r}$ on
$\mc M(\At)$ restricted to  $\mc M(\AT)$,  see Section \ref{sec comp}.
We also denote by  $ \frac \der{\der t_r}$ the $j$-th KdV vector field
on the space  $\D$, see  Section \ref{sec KdV}.

For a Miura map $\frak m_i : \mc M\to \D,\ \Ll \mapsto L_i,$ denote by $d\frak m_i$  the associated
derivative map $T\mc M(\At) \to T\D$.  By Theorem
\ref{thm mkdvtokdv} we have $d\frak m_i:  \frac \der{\der t_r}\big|_{\Ll}
\mapsto \frac \der{\der t_r}\big|_{L_i}$.

Fix a basic sequence $J=(j_1,\dots, j_m)$. Consider the associated family $\mu^J:\C^m\to \mc M(\AT)$
of Miura opers.
For a vector field $\Ga $ on $\C^m$, we denote by
$\frac {\der \mu^J}{\der \Ga}$ the derivative of $\mu^J$ along the vector field.
The derivative is well-defined since $\mc M(\AT)$ is an affine space.

\begin{thm}
\label{thm main}
Let $r>0$ and $r=1,5$ mod $6$. Then there exists a polynomial vector field $\Ga_r$ on $\C^m$
such that
\bean
\label{formula main}
\frac \der{\der t_r}\big|_{\mu^J(c)} = \frac {\der \mu^J}{\der \Ga_r}(c)
\eean
for all $c\in\C^m$. If $m$ is even and $r> 3m$, then $\frac \der{\der t_r}\big|_{\mu^J(c)}=0$
for all $c\in\C^m$ and, hence, $\Ga_r=0$. If $m$ is odd and $j_1=j_m=0$, then for $r> 3m-2$ we have
$\frac \der{\der t_r}\big|_{\mu^J(c)}=0$
for all $c\in\C^m$ and, hence,  $\Ga_r=0$. If $m$ is odd and $j_1=j_m=1$, then for $r> 3m+1$ we have
$\frac \der{\der t_r}\big|_{\mu^J(c)}=0$
for all $c\in\C^m$ and, hence, $\Ga_r=0$.

\end{thm}

\begin{cor}
\label{cor Main}
The family $\mu^J$ of Miura opers is invariant with respect to all mKdV flows of type $\AT$ and
is point-wise fixed by flows with $r>3m+1$.

\end{cor}

\subsection{Proof of Theorem \ref{thm main} for $m=1$}

If $m=1$, then $J=(0)$ of $J=(1)$.

Let $J=(0)$.  Then
\bea
\mu^J(c_1) &=& e^{g_1F_{0}}\Ll^\emptyset e^{-g_1F_0}
= (1 + g_1e_{3,3}\La^{-1})\Ll^\emptyset (1 - g_1e_{3,3}\La^{-1})
=
\\
&=&
\der + \La + g_1(e_{3,3}-e_{1,1})= \der + \La - g_1 h_0,
\eea
where $g_1 = \frac 1{x+c_1}$. By formula \Ref{T mkdv},
\bea
\frac \der{\der t_r}\big|_{\mu^J(c_1)}
= - \frac {\der}{\der x}((1 + g_1e_{3,3}\La^{-1}) \Lambda^r (1 - g_1e_{3,3}\La^{-1}))^0.
\eea
It follows from Lemma \ref{lem lambda} that $\frac \der{\der t_r}\big|_{\mu^J(c_1)}=0$ for $r>1$ and hence
 $\Ga_r=0$. For $r=1$, we have $\frac \der{\der t_1}\big|_{\mu^J(c_1)} = -\frac 1{(x+c_1)^2}(e_{1,1}-e_{3,3})$. On the other hand,
 $\frac\der{\der c_1}\mu^J(c_1)
 = -\frac{\der g_1}{\der c_1}h_0 = \frac 1{(x+c_1)^2}h_0$. Hence  $\Ga_1= -\frac \der{\der c_1}$.
Theorem \ref{thm main} is proved for $m=1$, $J=(0)$.

Let $J=(1)$.  Then
\bea
\mu^J(c_1) = e^{\ad g_1(2F_1+2F_2)}(\Ll^\emptyset)= \der + \La - g_1 h_1,
\eea
where $g_1 = \frac 1{x+c_1}$. By formula \Ref{T mkdv},
\bea
\frac \der{\der t_r}\big|_{\mu^J(c_1)}
=
&-& \frac {\der}{\der x}\Big((1 + g_1(e_{1,1}+e_{2,2})\La^{-1} +2g_1^2e_{1,1}\La^{-2})\times
\\
&\times &
 \Lambda^r(1 - g_1(e_{1,1}+e_{2,2})\La^{-1} +2g_1^2e_{1,1}\La^{-2})\Big)^0.
\eea
It follows from Lemma \ref{lem lambda} that $\frac \der{\der t_r}\big|_{\mu^J(c_1)}=0$ for $r>4$ and hence
 $\Ga_r=0$. For $r=1$, we have $\frac \der{\der t_1}\big|_{\mu^J(c_1)} = \frac 1{(x+c_1)^2}(e_{1,1}-e_{3,3})$. On the other hand,
$\frac\der{\der c_1}\mu^J(c_1)
 = -\frac{\der g_1}{\der c_1}h_1 = \frac 1{(x+c_1)^2}h_1$. Hence  $\Ga_1= -\frac 12\frac \der{\der c_1}$.
Theorem \ref{thm main} is proved for $m=1$, $J=(1)$.

\subsection{Proof of Theorem \ref{thm main} for $m>1$}
\begin{lem}
\label{lem r>2m}
If $m$ is even and $r> 3m$, then $\frac \der{\der t_r}\big|_{\mu^J(c)}=0$
for all $c\in\C^m$ and, hence, $\Ga_r=0$. If $m$ is odd and $j_1=j_m=0$, then for $r> 3m-2$ we have
$\frac \der{\der t_r}\big|_{\mu^J(c)}=0$
for all $c\in\C^m$ and, hence,  $\Ga_r=0$. If $m$ is odd and $j_1=j_m=1$, then for $r> 3m+1$ we have
$\frac \der{\der t_r}\big|_{\mu^J(c)}=0$
for all $c\in\C^m$ and, hence, $\Ga_r=0$.

\end{lem}

\begin{proof}
 The vector
 $\frac \der{\der t_r}\big|_{\mu^J(c)}$ equals the right hand side of formula  \Ref{T mkdv}.
By Lemmas \ref{lem exp} and \ref{lem lambda} the right hand side of \Ref{T mkdv} is zero if $r$ is as described in the lemma.
\end{proof}

We prove the first statement of Theorem \ref{thm main} by induction on $m$.
Assume that the statement is proved for ${\tilde J}=(j_1,\dots,j_{m-1})$.
Let
\bea
\label{J'2}
Y^{{\tilde J}}\ : \
{\tilde c}=(c_1,\dots,c_{m-1}) \ \mapsto \ (y_0(x,{\tilde c}), y_1(x,{\tilde c}))
\eea
be the generation of pairs in the ${\tilde J}$-th direction. Then the generation
of pairs in the $J$-th direction is
\bea
\label{J2}
Y^{J}\ :\ \C^m \mapsto (\C[x])^2, \quad
({\tilde c},c_m) \mapsto
\ (\dots,  y_{j_m,0}(x,{\tilde c}) + c_m y_{j_m}(x,{\tilde c}),\dots,
),
\eea
see \Ref{J'} and \Ref{Ja}.
We have $g_m = \ln'(y_{j_m,0}(x,{\tilde c}) + c_m y_{j_m}(x,{\tilde c}))- \ln' (y_{j_m}(x,{\tilde c}))$,
see \Ref{g's}.

By the induction assumption,
there exists a polynomial vector field $\Gamma_{r,{\tilde J}}=\sum_{i=1}^{m-1}\ga_i(\tilde c)\frac\der{\der c_i}$ on $\C^{m-1}$ such that
\bean
\label{fromula main m-1}
\frac \der{\der t_r}\big|_{\mu^{{\tilde J}}({\tilde c})} = \frac {\der \mu^{{\tilde J}} }{\der \Ga_{r,{\tilde J}}}({\tilde c})
\eean
for all ${\tilde c}\in\C^{m-1}$.

\begin{thm}
\label{prop ind}
There exists a scalar polynomial
$\ga_{m}(\tilde c,c_m)$ on $\C^m$  such that the vector field
$\Ga_r = \Gamma_{r,{\tilde J}} + \ga_{m}({\tilde c},c_m)\frac \der{\der c_m}$
satisfies \Ref{formula main} for all $(\tilde c,c_m)\in\C^m$.
\end{thm}

The first statement of Theorem \ref{thm main} follows from Theorem \ref{prop ind}.

\subsection{ Proof of Theorem \ref{prop ind}}
\label{sec proof prop ind}

\begin{lem}
\label{lem ker 1}
If $j_m=1$, then for all $({\tilde c},c_m)\in\C^m$, we have
\bean
\label{formula for ker j_m=1}
d\frak m_0\big|_{\mu^J({\tilde c},c_m)} \left(\frac{\der}{\der t_r}\big|_{\mu^J({\tilde c},c_m)} -
\frac{\der\mu^J}{\der \Ga_{r,{\tilde J}}}({\tilde c},c_m)\right) =0\,,
\eean
If $j_m=0$, then for all $({\tilde c},c_m)\in\C^m$, we have
\bean
\label{formula for ker j_m=0}
d\frak m_1\big|_{\mu^J({\tilde c},c_m)} \left(\frac{\der}{\der t_r}\big|_{\mu^J({\tilde c},c_m)} -
\frac{\der\mu^J}{\der \Ga_{r,{\tilde J}}}({\tilde c},c_m)\right) =0\,.
\eean
\end{lem}

\begin{proof}
The proof of this lemma is the same as the proof of Lemma 5.5 in \cite{VW}.
\end{proof}

Let $j_m=1$. By Lemma \ref{lem ker 1}, the vector $\frac{\der}{\der t_r}\big|_{\mu^J({\tilde c},c_m)} -
\frac{\der\mu^J}{\der \Ga_{r,{\tilde J}}}({\tilde c},c_m)$ lies in the kernel of
the map $d\frak m_0\big|_{\mu^J({\tilde c},c_m)}$. By Lemma \ref{lem kernel}, this kernel is generated
by  $\frac{y_0(x,\tilde c,m-1)}{y_1(x,\tilde c,c_m)^2}h_0$. By Lemma \ref{lem der c-m},
we have
$\frac{\der \mu^J}{\der c_m}({\tilde c},c_m)  = a
\frac{y_0(x,{\tilde c},m-1)}
{y_1(x,{\tilde c},c_m,m)^2} \,h_{0}
$ for some $a\in\C^\times$.  Hence there exists a number $\gamma_m(\tilde c,c_m)$ such that
$\frac{\der}{\der t_r}\big|_{\mu^J(\tilde c,c_m)} = \Ga_{r,\tilde J}\big|_{\tilde c} + \gamma_m(\tilde c,c_m)\frac\der{\der c_m}$.

Let $j_m=0$. By Lemma \ref{lem ker 1}, the vector $\frac{\der}{\der t_r}\big|_{\mu^J({\tilde c},c_m)} -
\frac{\der\mu^J}{\der \Ga_{r,{\tilde J}}}({\tilde c},c_m)$ lies in the kernel of
the map $d\frak m_1\big|_{\mu^J({\tilde c},c_m)}$. By Lemma \ref{lem kernel}, this kernel is generated
by the polynomial $\frac{y_1(x,\tilde c,m-1)^4}{y_0(x,\tilde c,c_m)^2}h_0$. By Lemma \ref{lem der c-m},
we have
$\frac{\der \mu^J}{\der c_m}({\tilde c},c_m)  = a
\frac{y_{1}(x,{\tilde c},m-1)^4}
{y_{0}(x,{\tilde c},c_m,m)^2} \,h_{0}
$ for some $a\in\C^\times$.  Hence there exists a number $\gamma_m(\tilde c,c_m)$ such that
$\frac{\der}{\der t_r}\big|_{\mu^J(\tilde c,c_m)} = \Ga_{r,\tilde J}\big|_{\tilde c} + \gamma_m(\tilde c,c_m)\frac\der{\der c_m}$.

\begin{prop}
\label{prop polyn}
The function $\ga_{m}({\tilde c},c_m)$ is a polynomial on $\C^m$.

\end{prop}

\begin{proof}

The proof is similar to the proof of Proposition 5.9 in \cite{VW}.

Let $g = x^d + \sum_{i=0}^{d-1} A_i(c_1,\dots,c_m) x^i$ be a polynomial in $x,c_1,\dots,c_m$. Denote
$h = \ln'g$ the logarithmic derivative of $g$ with respect to $x$. Consider the Laurent expansion
of $h$ at $x=\infty$, \
$h = \sum_{i=1}^\infty B_i(c_1,\dots,c_m) x^{-i}$.

\begin{lem}
\label{ lem Laurent}
All coefficients $B_i$ are polynomials in $c_1,\dots,c_m$.
\qed
\end{lem}

The vector $Y = \frac{\der}{\der t_r}\big|_{\mu^J({\tilde c},c_m)}$ is a $3\times 3$ diagonal matrix depending
on $x,c_1,\dots,c_m$, $Y=Y_1(e_{1,1}-e_{3,3})$ where $Y_1$ is a scalar function.

\begin{lem}
\label{ lem  dt Laurent}
The function $Y_1$ is a rational function in $x,c_1,\dots,c_m$ which  has a Laurent expansion of the form
$Y_1 = \sum_{i=1}^\infty B_i(c_1,\dots,c_m) x^{-i}$ where all coefficients $B_i$ are polynomials in $c_1,\dots,c_m$.
\qed
\end{lem}

The vector $Y = \frac{\der \mu^J}{ \der\Ga_{r,\tilde J} }({\tilde c},c_m)$ is a $3\times 3$ diagonal matrix depending
on $x,c_1,\dots,c_m$, $Y=Y_1(e_{1,1}-e_{3,3})$ where $Y_1$ is a scalar function.

\begin{lem}
\label{ lem  mu Laurent}
The function $Y_1$ is a rational function of $x,c_1,\dots,c_m$ which  has a Laurent expansion of the form
$Y_1 = \sum_{i=1}^\infty B_i(c_1,\dots,c_m) x^{-i}$ where all coefficients $B_i$ are polynomials in $c_1,\dots,c_m$.
\qed
\end{lem}

Let us finish the proof of Proposition \ref{prop polyn}. The function  $\ga_{m}({\tilde c},c_m)$ is determined from the
equation
\bea
\frac{\der}{\der t_r}\big|_{\mu^J({\tilde c},c_m)} -
\frac{\der\mu^J}{\der \Ga_{r,{\tilde J}}}({\tilde c},c_m) = a_1 \ga_{m}({\tilde c},c_m)\frac{y_0(x,{\tilde c}, m-1)}
{y_{1}(x,{\tilde c},c_m,m)^2}\, h_0
\eea
if $j_m=1$ and from the equation
\bea
\frac{\der}{\der t_r}\big|_{\mu^J({\tilde c},c_m)} -
\frac{\der\mu^J}{\der \Ga_{r,{\tilde J}}}({\tilde c},c_m) = a_2 \ga_{m}({\tilde c},c_m)\frac{y_1(x,{\tilde c}, m-1)^4}
{y_0(x,{\tilde c},c_m,m)^2}\, h_0
\eea
if $j_m=0$. Here $a_1,a_2$ are nonzero complex numbers independent of $\tilde c,c_m$.

The function  $\frac{y_0(x,{\tilde c}, m-1)}
{y_{1}(x,{\tilde c},c_m,m)^2}$ has the Laurent expansion of the form $\sum_{i=1}^\infty B_i(c_1,\dots,c_m) x^{-i}$ and the first nonzero coefficient
$B_i$ of this expansion is 1 since the polynomials $y_0, y_1$ are monic polynomials. Hence $\ga_m$ is a polynomial if $j_m=1$.
Similarly, the function  $\frac{y_1(x,{\tilde c}, m-1)^4}
{y_0(x,{\tilde c},c_m,m)^2}$ has the Laurent expansion of the form $\sum_{i=1}^\infty B_i(c_1,\dots,c_m) x^{-i}$ and the first nonzero coefficient
$B_i$ of this expansion is 1 since the polynomials $y_0, y_1$ are monic polynomials. Hence $\ga_m$ is a polynomial if $j_m=0$.
\end{proof}

Theorem \ref{thm main} is proved.

\subsection{Critical points and the population generated from $y^\emptyset$}

\begin{thm}
[\cite{MV3}] If a pair of polynomials $(y_0,y_2)$ represents a critical point of the master function
\Ref{Master} for some parameters $k=(k_0,k_1)$, then $(y_0,y_1)$ is a point
of the population of pairs  generated  from $y^\emptyset$.

\end{thm}

\end{document}